\documentclass[a4paper,12pt]{amsart}
\usepackage{graphicx}
\usepackage{amssymb,amsmath,latexsym,amscd}
\usepackage[matrix,arrow,curve]{xy}
\usepackage{hyperref}
\usepackage{verbatim}
\usepackage{url}
\theoremstyle{plain}
\newtheorem{thm}{Theorem}[section]

\newtheorem{lem}[thm]{Lemma}
\newtheorem{prop}[thm]{Proposition}
\theoremstyle{definition}
\newtheorem{defn}[thm]{Definition}
\theoremstyle{remark}
\newtheorem{rem}[thm]{Remark}
\newtheorem{exmpl}{Example}[section]

\newcommand{\abs}[1]{\left\vert#1\right\vert}
\newcommand{\set}[1]{\left\{#1\right\}}

\newcommand{\To}{\longrightarrow}

\newcommand{\X}{\mathfrak{X}}
\newcommand{\R}{\mathcal{R}}
\newcommand{\Z}{\mathbb{Z}}
\newcommand{\F}{\mathcal{F}}
\newcommand{\RR}{\mathbf{R}}
\newcommand{\C}{\mathbb{C}}
\newcommand{\CC}{\mathcal{C}}

\newcommand{\CVR}{\mathcal{C}=(V,\mathcal{R})}
\begin{document}
\title[]{\textit{Semisimplicity of cellular algebras over positive characteristic fields}}%
\author{Reza Sharafdini}
\address{Department of Mathematics, Collage of science, Pusan National University}%
\email{sharafdini@pusan.ac.kr}%

\dedicatory{}%
\maketitle
\begin{abstract}
 In this paper, we investigate semisimplicity of cellular
algebras over positive characteristic fields. Our main result
shows that the Frame number of cellular algebras characterizes semisimplicity of it. In a sense, this is a generalization of
Maschke's theorem.
\end{abstract}
\section{intrudoction}
Cellular algebras are an object in algebraic combinatorics which were introduced by B. Yu. Weisfeiler and
A. A. Lehman as cellular algebras and independently by D. G. Higman as coherent algebras (see \cite{Higman1987} and \cite{Weisfeiler1976}).
 They are by definition matrix algebras over a ring which is closed under the Hadamard multiplication
and the transpose and containing the identity matrix and the all one matrix.
Note that according to E. Bannai and T. Ito \cite{Bannai1999}, a homogeneous coherent configuration is also called an association scheme (not necessarily commutative).
Clearly, the adjacency algebra of a coherent configuration (or scheme) is a cellular algebra. Conversely, for each cellular algebra $W$  there exists a coherent configuration whose adjacency algebra coincides with $W$. So we prefer to deal with the adjacency algebra
of a coherent configuration. In a sense, cellular algebras are generalization of group algebras, so it is natural to extend Maschke's theorem (see \cite{Maschke1898} and \cite[Theorem III.1.22]{Nagao}) to them. Also E. Bannai and T. Ito in \cite[page 303]{Bannai1986},
asked about determination by the parameters, association schemes and fields
for which the adjacency algebras are semisimple, symmetric, Frobenius and quasi-Frobenius. We will answer this question about semisimplicity,
for general case, cellular algebras. In order to do this, we use the Frame number of cellular algebras. This number which was introduced  by J. S. Frame in 1941, is in relation with the
double cosets of finite groups. In 1976  D. G. Higman  extended this number to cellular algebras. Z. Arad in 1999  with  the help of Frame number characterized semisimplicity of
commutative cellular algebras (or commutative association schemes) over fields of prime order (see \cite{Arad1999}).
 Finally, A. Hanaki  in 2002 generalized the result by Z. Arad for homogeneous
cellular algebras (or association schemes\footnote{After P. Delsart a commutative scheme is called an association scheme, however the latter term introduced by
K. R. Nair originally is referred to a symmetric scheme.}) over  positive characteristic fields (see \cite{Hanaki2000}).\\
In this paper, we consider cellular algebras, not necessarily homogeneous,
over positive characteristic fields. Actually we  prove that a cellular algebra over a field $k$  is semisimple
if and only if its Frame number  is not divided
by characteristic of $k$.
\section{Definition and Notation}
To make this paper self-contained we put in this section the notations and definitions concerning  cellular algebras. For  more details, we refer to \cite{Inp2004}.
\begin{defn}
Let $V$ be a finite set and $\R$ a set of nonempty binary
relations on $V$. A pair $\CC = (V,\R)$ is called a \textit{coherent
configuration} or a \textit{scheme} on $V$ if the following conditions are
satisfied:
\begin{itemize}
\item[(C1)] $\R$ forms a partition of the set $V^2$.
\item[(C2)] the diagonal $\Delta(V)$ of $V^2$ is a union of  elements of $\R$.
\item[(C3)] for every $R\in\R$,\quad$R^t:=\set{(v,u):(u,v)\in R}\in \R$.
\item[(C4)] for every $R,S,T\in\R$, the number $|\{v\in V:\ (u,v)\in R,\ (v,w)\in S\}|$ does not depend on the
choice of $(u,w)\in T$ and  is denoted by $c_{R,S}^{T}$.
\end{itemize}
\end{defn}
\noindent The elements of $V$, the relations of $\R = \R(\CC)$ and the
numbers from condition (C4) are called the \textit{points}, the
\textit{basis relations} and the \textit{intersection numbers} of $\CC$,
resp. The numbers $deg(\CC)=|V|$ and $rk(\CC)=|\R|$ are
called the \textit{degree} of $\CC$ and the \textit{rank} of $\CC$, resp. Also $\R^\ast(\CC)$ is defined as the set of all \textit{relations} of $\CC$ each of which is a union of elements
of $\R(\CC)$.
\begin{exmpl}Let $G\leq Sym(V)$ be a permutation group and $\R=Orb_2(G)$  be the set of orbitals of $G$.
Then $\R$ forms a partition of the set $V^2$ such that
$R^t$  belongs to $\R$ for all $R\in
\R$. Moreover, the reflexive relation $\Delta(V)$ is a union of elements of $\R$. Finally, given
$(u, v)\in V^2$ and $R, S\in \R$, if we set $$p_{u,v}(R, S)=\set {v\in V
:(u, v)\in R,(v,w)\in S},$$ then obviously
$p_{u^g,v^g}(R^g,S^g)=p_{u^g,v^g} (R,S)$ for all $g\in G$. So the
number $|p_{u,v}(R, S)|$ does not depend on the choice of the pair
$(u,v)\in T$ for all $T\in \R$. Thus Inv$(G):=(V,\R)$ is a scheme which is called \textit{Schurian}.
Also Inv$({id_V})$ is called  the \textit{trivial scheme}(see \cite{Inp2004}).
\end{exmpl}
{\textbf{Adjacency algebra.} Let $\CC = (V,\R)$ be a scheme and
$\Z$ be the ring of integers. Given a relation $R\in \R$.
Denote by $A(R)$ the \emph{adjacency matrix} of $R$: $A(R)$ is
a $\{0,1\}$-matrix of the full matrix algebra $\text{Mat}_V(\Z)$ such that $A(R)_{u,v} = 1$ iff $(u, v)\in R$. Then from the
definition of $\CC$ it follows that the
$\Z$-linear span $W = W(\CC)$ of the set $\set{A(R) : R\in \R}$ in $Mat_V(\Z)$ satisfies the following conditions:\\
\begin{itemize}
    \item[(C$^{'}$1)]for every $R, S \in \R$,\quad\quad$A(R)A(S)=\displaystyle\sum_{T\in\R}c_{R,S}^{T}A(T)$,
    \item[(C$^{'}$2)]  $I_V , J_V \in W$,
    \item[(C$^{'}$3)]  $W$ is closed under the Hadamard (componentwise) multiplication,
    \item[(C$^{'}$4)]  $W$ is closed under the transpose map,
\end{itemize}
where $I_V$ is the identity matrix and $J_V$ is the matrix
whose  all entries are ones. In particular, $W$ is a ring with
respect to the both multiplications with identities $I_V$ and
$J_V$, resp. It is called the \textit{adjacency ring} of
the scheme $\CC$.
If $\mathfrak{R}$ is a ring, then we define
$W^\mathfrak{R}:=\mathfrak{R}\bigotimes_{\Z}W$, the \textit{adjacency algebra} of $W$ over $\mathfrak{R}$. If $\Z$ is a subring of $\mathfrak{R}$,
then $W$ is a subalgebra of $W^\mathfrak{R}$. We call $\set{A(R) : R\in \R}$  the \textit{standard basis} of $W(\CC)$.\\
\begin{defn}An $\mathfrak{R}-$subalgebra $W$ of the algebra $\text{Mat}_V(\mathfrak{R})$ is called a \textit{cellular algebra} on $V$ if it
satisfies conditions (C$^{'}$2)-(C$^{'}$4).
\end{defn}
\begin{exmpl}The adjacency ring of the trivial scheme on $V$ coincides with $\text{Mat}_V(\Z)$ and
the adjacency ring of a scheme of rank $2$ has a standard basis
$\{I_V , J_V -I_V \}$. Given a group $G$ and $g\in G$ the mapping $x\mapsto xg$ is a permutation of $G$ and denoted by $g_{right}$.
The set $G_{right}=\set{g_{right}|g\in G}$ is a permutation group on $G$. Analogously, the group $G_{left}$ consists of
permutations $x\mapsto g^{-1}x, g\in G$. Let $\CC=$Inv$(G_{left})$ for a group $G$.
Each basis relation of $\CC$ is of the form $R_g = \{(x, xg) :x\in
G\}$ for some $g\in G$. We observe that $A(R_g)=P_g$ where
$P_g$ is the permutation matrix corresponding to $g_{right}$. Since obviously $P_gP_h = P_{gh}$ for
all $g, h\in G$, the mapping
$$\Z[G]\To W(\CC)$$$$g\mapsto P_g$$
induces an algebra isomorphism from the group ring $\Z[G]$ of the
group $G$ to the cellular ring $W(\CC)$ of the scheme $\CC$.
\end{exmpl}
\begin{rem}It is  well known that there exists a one-to-one correspondence between the set of  all cellular algebras on $V$ and
 the set of all schemes on $V$.
Due to this correspondence, we can use freely both the language of
matrices and the language of relations. In particular,
a scheme $\CC$ is commutative, if so is the adjacency algebra of
$\CC$. This is equivalent to the equalities $c_{R,S}^T =c_{S,R}^T$ for
all $R, S,T\in \R(\CC)$. It is easy to see that any symmetric
scheme is commutative (the scheme $\CC$ is called \textit{symmetric }if
each basis relation of it is symmetric, i.e., $R=R^t$ for all $R\in \R(\CC)$).\\
\end{rem}
\textbf{Cells and basis relations}.
Let $\CC =(V,\R)$ be a scheme. Set
$$\text{Cel}(\CC) = \{X\subseteq V:\Delta(X)\in\R\},\quad\quad \text{Cel}^\ast(\CC)= \{X\subseteq V:\Delta(X)\in\R^\ast\}.$$
Each element of the set $\text{cel}(\CC)$ is called a \textit{cell} (resp.
\textit{cellular set}) of the scheme $\CC$. For a permutation group $G$ we have
\begin{equation}\label{cell}
    \text{Cel(Inv}(G)) = \text{Orb}(G),\quad \text{Cel}^\ast(\text{Inv}(G)) = \text{Orb}^\ast(G),
\end{equation}
in  which Orb$(G)$ is the set of all orbits of $G$  on $X$ and $\text{Orb}^\ast(G)$ is the set of all invariant sets of $G$ on $X$.
For instance, the Cells of the trivial scheme on $V$ are exactly
the singletons of $V$ , whereas a scheme of rank 2 on $V$ has the unique
cell, namely, $V$. Let $X, Y\in \text{Cel}^\ast(\CC)$. Then $\Delta(X)$ and
$\Delta(Y)$ are relations of the scheme $\CC$ and so the
adjacency matrices $I_X$ and $I_Y$ of them belong to the algebra
$W = W(\CC)$. This implies that $I_XJ_VI_Y = A(X\times Y)\in W$
and hence $X\times Y\in \R^\ast(\CC)$. Morovere From (C2) it follows that $\Delta(V)=\bigcup_{X\in \text{Cel}(\CC)}\Delta(X)$.
So $V$ is the disjoint union of cells and we have,
\begin{equation}\label{1}
    \R(\CC)=\bigcup_{X,Y\in \text{Cel}(\CC)} \R_{X,Y}\quad\quad(\text{disjoint union}),
\end{equation}

where for $X, Y\in \text{Cel}(\CC)$ we set $$\R_{X,Y} = \R_{X,Y}
(\CC) = \set{R\in\R:R\subseteq X\times Y }$$ For $R\in \R_{X,Y}$
with $X, Y\in \text{Cel}(\CC)$, set
\begin{equation}\label{1}
 d_{out}(R) = c_{R,R^t}^{\Delta(X)},\quad\quad d_{in}(R) =c_{R^t,R}^{\Delta(Y)}.
 \end{equation}
If $A = A(R)$, then $d_{out}(\R)$ (resp. $d_{in}(\R)$) is the number of ones in each row $u$ (resp. each column
$v$) of the matrix $A$ where $u\in X$ (resp. $v\in Y$). From the
definition of intersection numbers it follows that given $(u,
v)\in X\times Y$ we have
$$d_{out}(R) = |R_{out}(u)|,\quad\quad d_{in}(R) = |R_{in}(v)|,$$
where $R_{out}(u)=\set{w\in V : (u,w)\in R}$ and
$R_{in}(v)=\set{w\in V : (w, v)\in R}$. Thus
\begin{equation}\label{valency}
\sum_{R\in\R_{X,Y}}d_{out}(R)=|Y|,\quad\quad\sum_{R\in
\R_{X,Y}}d_{in}(R)=|X|,
\end{equation}
\begin{equation}\label{in&out}
    |X|d_{out}(R)=|R|=|Y|d_{in}(R).
\end{equation}
A scheme $\CC$ is called \textit{homogeneous} or (\textit{ an association scheme}) if $|\text{Cel}(\CC)|=1$
or equivalently, if $\Delta(V)\in \R$. (From (\ref{cell}) it
follows that for a permutation group $G$ the scheme Inv$(G)$ is
homogeneous iff the  group $G$ is transitive.) In this case  for given
$R\in\R$ we have
$$|V|d_{out}(R)=|R|=|V|d_{in}(R),\quad\quad d_{out}(R)=d_{in}(R).$$
The latter number is denoted by $d(R)$ and is called the degree of
the relation $R$. Thus each basis relation of a homogeneous scheme
can be treated as the set of arcs of a regular digraph with the
vertex set V . From (\ref{valency}) it follows that
\begin{equation}\label{calencyass}
\sum_Rd(R)=|V|.
\end{equation}
Suppose that $X\in \text{Cel}(\CC)$ and denote  by $I_X$ the adjacency matrix of $\Delta(X)$, then $I_V=\sum_{X\in \text{Cel}(\CC)}I_X$
is an idempotent decomposition of $I_V$.
We observe that every commutative scheme is homogeneous . (Indeed,
the commutativity of $\CC$ means the commutativity of the
adjacency algebra $W(\CC)$. If $ X, Y\in \text{Cel}(\CC)$ and $X\neq
Y$, then $I_XA(R)=A(R)$ and $I_XA(R^t)= 0$ for all
$R\in\R_{X,Y}$.)\\

\begin{defn}\label{Definition:dsm}
Let $\CC = (V,R)$ be a scheme with its adjacency ring $W$. If $\Z V$ is a free $\Z$-module of rank $|V|$
indexed by $V$, then $W$ acts naturally on the basis set $V$, namely $\Z V$ has the structure of a module over $\text{Mat}_{V}(\Z)$ according to
$$uA:=\sum_{v\in V} A_{u,v}v\quad\quad (A\in \text{Mat}_{V}(\Z), u\in V).$$
Assume that $F$ is a field  and define $FV:=F\bigotimes_{\Z}\Z V$. Then $FV$ can be regarded as $W^F\hspace{-1.5mm}-$module.
We call this  the \textit{standard module} of $W$(resp. $\CC$) over $F$. The character
of $W^F$ afforded by the standard module is called the \textit{standard character}
of $W$. We shall denote the standard character of $W$  by $\rho$ which is calculated in the following lemma (By $\delta$ we mean the Kronecker delta).
\end{defn}
\begin{lem}\label{lemma:standard charcter}For every $R\in \R$ we have~~$\rho(A(R))=\sum_{X\in \text{Cel}(\CC)}\delta_{R,\Delta(X)}\abs{X}$.
\end{lem}
We state here some facts about finite dimensional algebras. Let $A$ be a finite dimensional algebra over $F$
(concerning finite dimensional algebras we refer to \cite{Nagao}). The
\textit{Jacobson radical} Rad($A$) of $A$ is the intersection of all maximal right ideals of $A$.
Also $A$ is said to be \textit{semisimlpe} if Rad$(A)=0$. In section \ref{Section:Discriminant}
we introduce another criterion for semisimplicity of finite dimensional algebras.\\
Let $K$ be an extension field of $F$. Then Rad$(A)\bigotimes_F K\subseteq \text{Rad}(A\bigotimes_F K)$, since Rad$(A)\bigotimes_F K$ is a nilpotent ideal of $A\bigotimes_F K$.
However, they do not necessarily coincide. But if $K$ is a separable extension of $F$, then the equality holds. Also $A$ is called \textit{separable}
 over $F$ if $A$ is semisimple and  $A\bigotimes_F K$ remains semisimple for any extension $K$ of $F$.
\begin{thm}\cite[Theorem II.5.4]{Nagao}\label{Theorem:perfect}
If $F$ is a perfect field (e.g., char$(F)=0$ or $F$ is finite), then every semisimple $F$-algebra is separable over $F$.
\end{thm}
We denote the complete set of representatives of
isomorphim classes of irreducible $A$-modules by IRR($A$).\\
It is well known that $A/\text{Rad}(A)$ is semisimple. If $A$  is a split $F$-algebra, namely $F$ is an splitting field for $A$, we have
$$A/\text{Rad}(A)=\bigoplus_{i=1}^{r}M_{f_i}(F),$$ where $f_i$'s  are the degrees of irreducible representations of $A$. So we have the following
(see \cite{Nagao} for details).
\begin{prop}
Let $A$  be a split $F$-algebra with \text{IRR}$(A)=\set{M_1, . . . , M_r}$ . Then
\begin{equation*}
\text{dim}_F(A)=\sum_{i=1}^{r}(deg~M_i )^2 + dim_F \text{Rad}(A).
\end{equation*}

\end{prop}
Let $W$ be a cellular algebra  and $k\leq K$ be fields. Then there is a natural isomorphism $W^k\bigotimes_kK\simeq W^K$ of $K$-algebras such that
$\alpha\otimes x\mapsto \alpha x$. Thus $W^K$ is just the scalar extension of $W^k$ for every subfield $k$ of $K$. This is quite useful in  the study
of cellular algebras. By using this we can prove the following result.
\begin{lem}\cite[Lemma III.1.28]{Nagao}\label{separable algebra}
If $W^K$ is a cellular algebra over field $K$, then  $W^K/Rad(W^K)$ is separable over $K$.
\end{lem}
\section{Discriminant of algebras}
\label{Section:Discriminant}
Let $\mathfrak{R}$ be a principal ideal domain, and $A$ a free $\mathfrak{R}$-algebra of finite rank $n$. Suppose that $M$ is a finite-dimensional
$A$-module with  a matrix representation $\X$, we define the \textit{discriminant} of the representation module $M$ as follows. The map
$\Phi_M:A\times A\To \mathfrak{R}$ defined by
$$\Phi_M(a,b)=\text{Tr}(\X(ab)),$$ is a symmetric bilinear
form, where Tr is the usual trace of matrices. Let ${a_1,\dots,
a_n}$ be an $\mathfrak{R}$-basis of $A$. We put
$$D_{M,\{a_i\}}(A) =\text{det(Tr}(\X(a_ia_j))).$$
Especially, when the representation $\X$ is the regular
representation, we call\\ $D_{M,\{a_i\}}(A)$ the discriminant of A,
and denote it by $D(A)$.
Note that $D_{M,\{a_i\}}(A)\neq0$ iff $\Phi_{M}$ is nondegenerate.
We note  that $D_{M,\{a_i\}}(A)$ depends on the choice
of the basis $\{a_i\}$ of $A$, but being nondegenerate is
independent on it. i.e., if we take another basis ${a'_1 ,\dots,
a'_r }$, then det(Tr$(\X(a'_ia'_j))) ={\varepsilon}^2
\text{det(Tr}(\X(a_ia_j)))$ for some unit $\varepsilon$ in $\mathfrak{R}$.
Hence, if $\mathfrak{R}=\Z$, then the
discriminant is uniquely determined.\\

Suppose $A$ is not semisimple. Then Rad$(A)\neq0$. If $0\neq a\in \text{Rad}(A)$, then $\Phi_M(a,b)=0$ for any $b\in A$, since each element of Rad$(A)$ is nilpotent. So $\Phi_{M}$ is degenerate.\\
Conversely, assume that $A$ is a semisimple split algebra and
 IRR$(A)=\set{M_1,\dots,M_r}$.
We have $W^{K}\simeq \bigoplus_{i=1}^{r}M_{f_i}(K)$, since $W^{K}$ is a split $K$-algebra . We
consider another basis $B=\set{e_{st}^{(i)}\mid 1\leq i\leq r,\quad 1\leq s,t\leq f_i}$ of $W^{K}$, where $e_{st}^{(i)}$ is
matrix unit in $M_{f_i}(K)$. If we put $M=\bigoplus_{i=1}^{r}M_i$, then
$$\Phi_{M}(A)=\bigoplus_{i=1}^{r}\Phi_{M_i}(A),\quad\quad D_{M,B}(A)=\prod_{i=1}^{r}D_{M_i,B_i}(A),$$
where $B_i=\set{e_{st}^{(i)}\mid1\leq s,t\leq f_i}$ and $\Phi_{M}(A)$ is the direct sum of $\Phi_{M_i}(A)$ for $1\leq i\leq r$ . We may assume that $A=M_{f_i}(K)$ where $f_i$ is the degree of $M_i$.
Given $0\neq a \in A$ with
nonzero entry $a_{ij}$, then
\begin{eqnarray*}
\Phi_{M}(a,e_{ji})&=&\text{Tr}(ae_{ji})=\sum_{t=1}^{f_i}(ae_{ji})_{tt}\\
&=&\sum_{t=1}^{f_i}\sum_{k=1}^{f_i}a_{tk}(e_{ji})_{kt}\\
&=&\sum_{k=1}^{f_i}a_{ik}(e_{ji})_{ki}=a_{ij}\neq0,
\end{eqnarray*}
where $e_{ji}$ is matrix unit. Hence $A$ is nondegenerate.
\begin{thm}\label{Theorem:Reza1}
Let $\CVR$ be a scheme with  standard basis $\set{A_i}$. If $K$ is a field  and $N=KV$ is the standard
module of $\CC$, then
$$D_{N,\set{A_i}}(W^{K})=\varepsilon\prod_{R\in\R}|R|;\quad\quad\varepsilon\in\set{1,-1}.$$
\end{thm}
\begin{proof}
Assume that $\rho$ is the standard character of $\CC$. Then by lemma \ref{lemma:standard charcter} and (\ref{1})  we have

\begin{eqnarray*}
\Phi_{N}\Bigl(A(R),A(S)\Bigr)
&=&\rho\Bigl(A(R)A(S)\Big)\\
&=&\sum_{T\in\R}c_{R,S}^{T}\rho(A(T))\\
&=&\sum_{T\in\R}\sum_{X\in \text{Cel}(\CC)}c_{R,S}^{T}\delta_{T,\Delta(X)}\abs{X}\\
&=&\delta_{S,R^t}c_{R,S}^{\Delta(X)}\abs{X}\quad\quad (R\subseteq X\times Y)\\
&=&\delta_{S,R^t}d_{out}(R)\abs{X}
=\delta_{S,R^t}\abs{R}.
\end{eqnarray*}
So
\begin{eqnarray*}
D_{N,\set{A_i}}(W^{K})
&=&\sum_{\sigma\in Sym(\R)}sgn(\sigma)\prod_{R\in\R}\Phi_{N}(A(R),A(R^\sigma))\\
&=&\sum_{\sigma\in Sym(\R)}sgn(\sigma)\prod_{R\in\R}\delta_{R^\sigma,R^t}\abs{R}\\
&=&\varepsilon\prod_{R\in\R}|R|;\quad\quad\varepsilon\in\set{1,-1}.
\end{eqnarray*}
\end{proof}
\section{Frame number}}
In this section, we define the Frame number of a scheme. This
number was introduced by Frame \cite{Frame1941} and  was extended  to cellular algebras by
D.G. Higman \cite{{Higman1975}}.\\\\
Let $\CC=(V,\R)$ be a scheme and $W\leq \text{Mat}_V(\Z)$ the adjacency
ring of $\CC$. Suppose that IRR$(W^{\C})=\{M_1,\dots,M_r\}$, $f_i= dim(M_i)$ and $\C V$ is the standard module of $\CC$ over complex field $\C$.
As $W^{\C}$ is semisimlple (see Proposition \ref{Proposition: Semisimlicity char0}), then we have
$$\C V=\bigoplus_{i=1}^{r}m_iM_i.$$
We call $m_i$ the \textit{multiplicity }of $M_i$.

If we consider
matrix form of \textit{Schur relations} of $\CC$(see \cite{Higman1975}), then we obtain the
following
\begin{equation}\label{Fr}
\mathfrak{N}(\CC)=\frac{\prod_{X\in
\text{Cel}(\CC)}|X|^{-2}\prod_{R\in\R}|R|}{\prod_{i=1}^{r}{m_i^{f_i^2}}},
\end{equation}
where the number $\mathfrak{N}(\CC)$ is called the \textit{Frame Quotient} of $\CC$. It is  well known
that $\mathfrak{N}(\CC)$  is a rational integer (see \cite{Higman1975}). We define the \textit{Frame number} $\F(\CC)$ by
$$\F(\CC)=\frac{\prod_{R\in\R}|R|}{\prod_{i=1}^{r}m_i^{f_i^2}}.$$ It is clear that $\F(\CC)$ is  a rational integer.
This number is a criterion for the semisimplicity of
cellular algebras, which will appear in our main result.\\
\begin{thm}\label{Theorem:Hanaki2}
Let $K$ be a field of characteristic zero and $W^{K}$ be a  split cellular algebra over $K$. If  IRR $(W^{K})=\set{M_1,\dots,M_r}$ and
$M=\bigoplus_{i=1}^{r}M_i$, then
 $$D_{M,\set{A_i}}(W^{K})=\varepsilon\F(\CC) ;\quad\quad\varepsilon\in\set{1,-1}$$
where $\set{A_i}$ is the  standard basis basis  of $W$.
\end{thm}
\begin{proof}
We have $W^{K}\simeq \bigoplus_{i=1}^{r}M_{f_i}(K)$, since $W^{K}$ is a split $K$-algebra . We
consider another basis $B=\set{e_{st}^{(i)}\mid 1\leq i\leq r,\quad 1\leq s,t\leq f_i}$ of $W^{K}$. Let $P$ be the transformation matrix of the bases
$\set{A_i}$ and $B$. Then
$$D_{M,\set{Ai}}(W^{K})=D_{M,B}(W^{K})(\text{det} P)^2,$$
and we have $D_{M,B}(W^{K})=\pm1$ by
$\Phi_M(e_{st}^{(i)},e_{uv}^{(j)})=\delta_{ij}\delta_{sv}\delta_{tu}$.\\\\
Thus $D_{M,\set{Ai}}=\pm(\text{det} P)^2$. Next we put $N=KV$, the standard module. Then $N=\bigoplus_{i=1}^{r}m_iM_i$ and
$\Phi_N(e_{st}^{(i)},e_{uv}^{(j)})=\delta_{ij}\delta_{sv}\delta_{tu}m_i$. Thus
$$D_{N,\set{A_i}}(W^{K})=D_{N,B}(W^{K})(\text{det} P)^2=\pm\prod_{i=1}^{r}m_i^{f_i^2}(\text{det} P)^2.$$
By the Theorem \ref{Theorem:Reza1}, $D_{N,\set{A_i}}(W^{K})=\varepsilon\prod_{R\in\R}|R|$ where $\varepsilon\in\set{1,-1}$. Now we have
$$D_{M,\set{Ai}}(W^{K})=\pm(\text{det} P)^2=\pm\frac{\prod_{R\in\R}|R|}{\prod_{i=1}^{r}m_i^{f_i^2}}=\pm\F(\CC).$$
\end{proof}

\section{Semisimlicity}\label{Section: Semisimlicity}
Let $A\leq \text{Mat}_V(\C)$ and suppose that $A$ is closed under the complex conjugate transpose map. Then $A$ is semisimple.
Let $\mathcal{C}=(V,\mathcal{R})$  be a scheme with cellular ring $W$. Also suppose that $\{A_i\}$ is the
standard basis of $W$. By definition of schemes $W^\C$ is closed under the complex conjugate transpose map and thus it is semisimple.
\begin{prop}\label{Proposition: Semisimlicity char0}
Let $k$ be a field of characteristic $p$. Then in the above notation, the following hold:
\begin{enumerate}
    \item  If $p\nmid\prod_{R\in \R}\abs{R}$, then $W^{k}$  is
semisimple.
    \item If $p\mid\prod_{X\in \text{Cel}(\CC)}\abs{X}$, then $W^{k}$ is not semisimple.
\end{enumerate}
\end{prop}
\begin{proof}
For (1) see \cite[Theorem 4.1.3]{Zieschang1996}.
If $p\nmid\abs{X}$ for some $X$ in Cel$(\CC)$, then $P=\sum_{X\in Cel(\CC)}\prod_{Y\in Cel(\CC),Y\neq X}|Y|J_X$ is  a nonzero central nilpotent element
of $W^{k}$.Otherwise $|X|=p^{\alpha_{X}}m$, $(p,m)=1$. If we put $\lambda:=Max\set{\alpha_{X}|X\in \text{Cel}(\CC)}$,
then $p^{\lambda}\sum_{X\in Cel(\CC)}J_X/|X|$ is a nonzero central nilpotent element of $W^{k}$ and thus $W^{k}$ is not semisimple.

\end{proof}

\begin{thm}\cite[Hanaki]{Hanaki2000}\label{Theorem:Hanaki1}
Let $k$ be a field of characteristic p.
Suppose $\mathcal{C}=(V,\mathcal{R})$ is an association scheme. Then $W^{k}$
is semisimple if and only if the Frame number $\mathcal{F}(\CC)$
is not divided by $p$.
\end{thm}
The following is our main result and it is a generalization of
Theorem \ref{Theorem:Hanaki1}, in which we do not need that $W$ is
homogeneous.
\begin{thm}\label{Theorem:Reza2}
Let $k$ be a field of characteristic p. Then $W^{k}$ is semisimple
if and only if the Frame number $\mathcal{F}(\CC)$ is not divided
by $p$.
\end{thm}
In order to prove this result, we use the following tools which is quite useful in  the study  of modular representation theory.\\

\noindent Let $p$ be a prime, and let $(K,\RR, F)$ be a $p$-modular system for $W$.
Namely, $\RR$ is a complete discrete valuation ring with the maximal
ideal $(\pi)$, $K$ is the quotient field of $\mathbf{R}$ and its
characteristic is 0, and $F$ is the residue field $R/(\pi)$ and
its characteristic is $p$. For more \text{det}ails about $p$-modular
systems, see \cite{Nagao}. The simplest example of\\ $p$-modular
systems is $(\mathbb{Q},\mathbb{Z}_{(p)},\mathbb{Z}_p)$, in which $\mathbb{Z}_{(p)}$ is the localization of $\Z$ at  prime ideal $p\Z$. To
simplify our argument, we suppose that
$W^{K}$ and $W^{F}$ are split algebras. In this case, we say
$(K,\RR,F)$ is a splitting $p$-modular system of $W$. For $x\in
\RR$, we denote the image of the natural homomorphism $R\To F$ by $x^{\ast}$.\\
Each idempotent of $W^{F}$ is the image of an idempotent of $W^{\RR}$
by the natural epimorphism from $W^{\RR}$ to $W^{F}\simeq
W^{\RR}/\pi W^{\RR}$. The primitivity of idempotents is preserved
by this correspondence (see \cite[Theorem I.14.2]{Nagao}).
Moreover, there exists natural correspondence between the set of
central primitive  idempotents of $W^{\RR}$ and that of $W^{F}$ (see \cite[Proposition 1.12]{Dade1973}). Namely, if $$1 = e_0 + e_1 +\dots+e_r$$
 is the central idempotent decomposition of $1$ in $W^{\RR}$, then so is
$$1^{\ast}= e_0^{\ast} + e_1^{\ast} +\dots+ e_r^{\ast}$$ and we have the following proposition.

\begin{prop}\cite[Hanaki]{Hanaki2000}\label{Proposition:Hanaki3} Suppose $W^F$ is semisimple. Then there exists a natural
correspondence between IRR$(W^K)$ and IRR$(W^F)$ which preserves dimensions. Namely, If $\set{e_0 , e_1 ,\dots,e_r}$ is the set of
central primitive  idempotents of $W^K$, then so  is $\set{e_0^{\ast},e_1^{\ast}, \dots,e_r^{\ast}}$ of that in $W^F$.
\end{prop}
Let IRR$(W^{K})=\set{M_1,\cdots,M_r}$ and let $\mathfrak{X}_i$ be a matrix representation of $W^{K}$
corresponding to $M_i$. Put $f_i=dim M_i$. By \cite[Theorem II.1.6]{Nagao}, $M_i$ has an  $\RR$-form,
namely, we may assume that $\mathfrak{X}_i(A_j)\in M_{f_i}(\RR)$. Put $M=\bigoplus_{i=1}^{r}M_i$, and we can define an $W^\RR$-module $\widetilde{M}$
such that $K\otimes \widetilde{M}\simeq M$. Then we can define an $W^F$-module $M^{\ast}=\widetilde{M}/\pi M$. Obviously,
$\Bigl(D_{M,\set{A_i}}(W^{K})\Bigr)^\ast=D_{M^\ast,\{A_i^\ast\}}(W^F)$. By Proposition \ref{Proposition:Hanaki3}, if $W^F$ is semisimple, then
$\Bigl(D_{M,\set{A_i}}(W^{K})\Bigr)^\ast\neq0$. Also if $W^F$ is not semisimple, then $\Bigl(D_{M,\set{A_i}}(W^{K})\Bigr)^\ast=0$.

\noindent Therefor we can say that  $D_{M,\set{A_i}}(W^{K})$ characterizes the semisimplicity of $W^F$. Now by Theorem \ref{Theorem:Hanaki2}
$D_{M,\set{A_i}}(W^{K})=\pm\F(\CC)$ and thus Theorem \ref{Theorem:Reza2} holds for $F$. As the prime field of order $p$ is a perfect field, by Theorem \ref{Theorem:perfect}, Theorem \ref{Theorem:Reza2}
 holds for arbitrary field $k$ of characteristic $p$.\\

\noindent\textbf{Acknowledgements}\\
The author would like to thank  I. Ponomareno, M. Hirasaka, and  A. Hanaki
for giving the author their valuable and farsighted comments. The author would like to express his gratitude to his M.Sc. Supervisor
A. Rahnama, who encouraged him to study and do research in representation theory.


\begin{thebibliography}{100}

\bibitem{Arad1999}
Z. Arad, E. Fisman, and M. Muzychuk,``~\textit{Generalized table
algebras}", Israel J. Math. 144 (1999), 29-60.
\bibitem{Bannai1999}
E. Bannai and T. Ito,``~\textit{Algebraic Combinatorics. I. Association
Schemes}", Benjamin Cummings, Menlo Park, CA, (1984).
\bibitem{Bannai1986}
E. Bannai and T. Ito,``~\textit{Current research on algebraic combinatorics: supplements to our book}",
Algebraic combinatorics. I, Graphs Combin. 2 (1986), no. 4, 287-308.
\bibitem{Dade1973}
E. C. Dade,``~\textit{Block extensions}", Illinois J. Math. 17,
198-272, (1973).
\bibitem{Frame1941}
J. S. Frame,``~\textit{The double cosets of a finite Groups}",
Bull. Amer. Math. Soc. 47 (1941), 458-467.
\bibitem{Hanaki2000}
A. Hanaki,``~\textit{Semisimplicity of adjacency algebras of
association schemes}", {\it Journal of Algebra}, 225(2000),124-129.
\bibitem{Higman1975}
D. G. Higman,``~\textit{Coherent configurations. I. Ordinary
representation theory}", Geom. dedicato 4:1-32(1975).
\bibitem{Higman1987}
D. G. Higman,``~\textit{Coherent algebras}, Linear Algebra Appl., 93(1987), 209-239.
\bibitem{Inp2004}
I. Ponomarenko,``~\textit{Coherent Configurations and Permutation
Groups}", Lecturnote, Zanjan, Iran, (2004) preprint.
\bibitem{Maschke1898}
H. Maschke, Uber den arithmetischen Charakter der Coefficienten
der Subshtutionen endlicher linearer Substitutionsgruppen, Math. Ann. 50 (1898), 482-498.
\bibitem{Nagao}
H. Nagao and Y. Tsushima,``~\textit{Representations of Finite
Groups}", Academic Press, NewYork, (1989) Edited by Boris
Weisfeiler, Lecture I. Notes in Mathematics, Vol. 558.
\bibitem{Weisfeiler1976}
B. Weisfeiler,``~\textit{On construction and identification of
graphs}", Springer Lecture Notes, 558, (1976).
\bibitem{Zieschang1996}
P.-H. Zieschang,``~\textit{An Algebraic Approach to Association
Schemes}", Springer-Verlag, Berlin/NewYork, (1996).
\end{thebibliography}
\end{document}